\documentclass[10pt]{amsart}

\usepackage{amsmath,enumerate, xcolor,verbatim}
\usepackage{amsfonts}
\usepackage{tikz}
  \usepackage{pgffor}\usetikzlibrary{math}
\usepackage{amssymb}
\usepackage{amsxtra}
\usepackage{amsthm}
\usepackage{amsmath,amssymb,amsfonts,amsthm,color,latexsym}
\usepackage{xcolor}
\usepackage{float}

\newcommand{\Z}{\mathbb{Z}}

\newcommand{\C}{\mathbb{C}}
\newcommand{\F}{\mathcal{F}}

\newcommand{\B}{\mathcal{B}}

\newcommand{\ord}{\text{ord}}

\newtheorem{theorem}{Theorem}[section]
\newtheorem{example}[theorem]{Example}
\newtheorem{remark}[theorem]{Remark} 

\newtheorem{lemma}[theorem]{Lemma}
\newtheorem{proposition}[theorem]{Proposition}
\newtheorem{corollary}[theorem]{Corollary}
\newtheorem{definition}[theorem]{Definition}
\usepackage[pagebackref]{hyperref}
\newcommand{\mr}[1]{\ensuremath{\mathrm{#1}}}
\newcommand{\mc}[1]{\ensuremath{\mathcal{#1}}}
\newcommand{\mb}[1]{\ensuremath{\mathbb{#1}}}

\title{Indices of arbitrary holomorphic foliations in the plane}

\title{A Unified Numerical Framework for Classical Invariants in Holomorphic Foliations}

\title{Numerical Invariants of Holomorphic Foliation Germs}
\title{Numerical Invariants of Holomorphic Foliation Germs from reduction of singularities}

\title{Formulae for indices of holomorphic foliations via reduction of singularities}

\date{\today}

\author[M. Falla Luza]{Maycol Falla Luza}
\address[M. Falla Luza]{Instituto de Matem\'atica e Estat\'istica, Universidade Federal Fluminense, Rua Professor Marcos Waldemar de Freitas Reis, s/n, CEP 24210-201\\
Bloco H, Campus do Gragoat\'a - Niter\'oi - RJ, Brasil}
\email{hfalla@id.uff.br}

\author[P. Fernández-Sánchez]{Percy Fernández}
\address[P. Fernández-Sánchez]{Departamento de Ciencias, Secci\'on Matem\'aticas, Pontificia Universidad Cat\'olica del Per\'u, Av. Universitaria 1801, San Miguel, Lima, Per\'u}
\email{pefernan@pucp.edu.pe}

\author[D. Mar\'in]{David Mar\'in}
\address[D. Mar\'in]{Departament de Matem\`atiques, Universitat Aut\`onoma de Barcelona, Edifici C, 08193 Cerdanyola del Vall\`es, Barcelona, Spain}
\email{david.marin@uab.cat}



\thanks{ }

\begin{document}

\begin{abstract}
We study numerical invariants associated with the reduction of singularities of holomorphic foliation germs on $(\mathbb{C}^2,0)$. Building on our previous work on generalized curve foliations, we extend explicit formulas for several fundamental invariants to arbitrary foliations. In particular, we provide general expressions for the discrepancy vector, the Milnor and intrinsic Milnor numbers, and 
classical indices along a separatrix
as
Camacho-Sad, Variation, Gómez-Mont-Seade-Verjovsky and also the Baum-Bott index. These extensions require a careful analysis of the contributions of saddle-nodes arising in the desingularization process. As applications, we recover results of Brunella and Cavalier-Lehmann, as well as a related statement appearing in \cite{FP-GB-SM2021}, within a unified and purely numerical framework. Furthermore, we obtain intrinsic characterizations of generalized curve foliations in terms of indices and of second type foliations in terms of the discrepancy vector. 
\end{abstract}

\maketitle

\tableofcontents
\section*{Introduction}

The study of singular holomorphic foliations in the complex plane has long been guided by the search for numerical and geometric invariants capable of capturing subtle local and global behaviors. In a previous work \cite{FFMR}, we introduced explicit formulas for generalized curve foliations. 
These formulas, which relate the combinatorial data of the desingularization process to analytic properties of the foliation, provide the starting point for the present article.
In this paper, we extend those results to arbitrary germs of holomorphic foliations in $(\mathbb{C}^2,0)$. More precisely, we generalize the formulas for the discrepancy vector (Theorem~\ref{ell}), the Milnor number (Theorem~\ref{muF}), 
the intrinsic Milnor number and the Camacho-Sad, Variation, Gómez-Mont-Seade-Verjovsky indices along a separatrix, as well as the Baum-Bott index (Theorem~\ref{Indices}). These generalizations preserve
the explicit and computationally accessible nature of the original formulas. The key point in these extensions is the careful incorporation of the invariants associated with the saddle-nodes that appear during the reduction of singularities.
Some examples illustrate the scope and effectiveness of our methods.

As an application, we recover a theorem of Brunella and Cavalier-Lehmann (Proposition~\ref{LC}), together with a related result appearing in \cite[Proposition~4.7]
{FP-GB-SM2021} (Corollary~\ref{mil = intmil}). Our approach provides a unified framework for understanding these statements and highlights the essential role played by discrepancies and indices in their proofs.
In the final part of the article, we obtain new characterizations of generalized curve foliations in terms of indices (Theorem~\ref{Pi}). Furthermore, Theorem~\ref{ell} yields a characterization of second type foliations in terms of the discrepancy vector. These criteria offer new insights into the geometry of both families, establishing new connections between analytic, topological, and combinatorial aspects.

The paper is organized as follows.
Section~1 collects, for the convenience of the reader, the basic definitions concerning singularities of holomorphic foliation germs in the plane. Section~2 introduces the discrepancy vector and establishes its main properties, including the generalized formula for the Milnor number. Section~3 develops the generalized formulas for the various indices and examines their consequences. Section~4 contains the applications to known results and the characterization of generalized curve foliations. 
Finally we briefly discuss the dependence of our formulas on the ordering of the irreducible components of the exceptional divisor.

\section{Preliminaries on Holomorphic Foliations and Notation}

Let $\F$ be a germ of singular holomorphic foliation on $(\C^2,0)$ defined by a $1$--form $\omega$ or a vector field $\upsilon$. A formal curve $C=\{f=0\}$ is a \textbf{separatrix} of $\F$ if $f$ divides $df \wedge \omega$. When the curve is analytic we say that the separatrix is convergent. 

We say that $0\in \C^2 $ is a reduced singularity for $\F$ if the linear part $D\upsilon (0)$ of the vector field $\upsilon$ is non-zero and has eigenvalues $\lambda_1$, $\lambda_2$ fitting in one of the two cases:
\begin{enumerate}
\item $\lambda_1 \cdot \lambda_2 \neq 0$ and $\lambda_1 / \lambda_2 \notin \mathbb{Q}^+$, non-degenerate singularity;
\item $\lambda_1 \neq 0$ and $\lambda_2 =0$, saddle-node singularity.
\end{enumerate}

In the first case the foliation is given in some analytic coordinates by an equation of the form 
$$
\omega = x(\lambda_1 + \ldots)dy - y(\lambda_2 +\ldots)dx
$$
so there are exactly two separatrices $\{x=0\}$, $\{y=0\}$ throught the singularity. In the second case, up to a formal change
of coordinates, the singularity is given by a 1-form of the type
$$
\omega = x^k dy -y (1+\lambda x^{k-1})dx
$$
where $\lambda \in \C$ and $k \geq 2$. The curve $\{x = 0\}$ is a convergent separatrix, called strong, whereas $\{y = 0\}$ corresponds to a possibly
formal separatrix, called weak.

It is well known that there is always a reduction of singularities, that is, a finite composition of blow-ups $\pi:(M,E)\to (\C^2,0)$ such that all singularities of $\bar{\F}:=\pi^* \F$ are reduced (see, for example, \cite{CLS}). Moreover, there exists a minimal reduction of singularities, in the sense that any other reduction is obtained from it by an additional sequence of blow-ups.
Throughout this paper, $\pi$ will denote a (not necessarily minimal) reduction of singularities  of $\F$.

For a component $D$ of the exceptional divisor $E$, there are two possibilities:
\begin{enumerate}
\item $D$ is invariant by $\bar{\F}$ (non-dicritical). In this case, $D$ contains a finite number of reduced singularities. Each non-corner singularity carries a separatrix transversal to $D$.
\item $D$ is not invariant by $\bar{\F}$ (dicrital). In this case, by definition, D may intersect only non-dicritical components and $\bar{\F}$ is everywhere
transverse do $D$.
\end{enumerate}

A saddle-node singularity $q \in  Sing(\bar{\F})$ is is said to be a \textbf{tangent saddle-node} if its weak separatrix is contained in the exceptional divisor $D$. Observe that if a corner singularity is a saddle-node, it would be necessarily a tangent saddle-node.

\begin{definition}
Consider the following subsets of~$E$:
\begin{enumerate}
\item $SN(\bar\F)\subset E$: saddle-nodes,
\item $TSN(\bar\F)\subset SN(\bar\F)$: tangent saddle-nodes,
\item $CSN(\bar\F)\subset TSN(\bar\F)$: corner saddle-nodes.
\end{enumerate}
We say that
\begin{enumerate}
\item $\F$ is a \textbf{generalized curve} if $SN(\bar\F)=\emptyset$,
\item $\F$ is \textbf{of second type} if $TSN(\bar\F)=\emptyset$,
\item $\F$ is \textbf{corner-non-degenerate} (CND) if $CSN(\bar\F)=\emptyset$.
\end{enumerate}
\end{definition}

\begin{remark}\label{BSN}
If we blow-up a tangent (resp. non-tangent) saddle-node $q$ we obtain an ordinary saddle and a  corner (resp. non-tangent) saddle-node $q'$ with the same invariant $k$ whose weak separatrix is the strict transform of the weak separatrix of $q$.
In particular, the notion of CND is not intrinsic in the sense that it depends on the reduction of singularities. What is well-defined is the CND condition for the minimal reduction of singularities.
\end{remark}

\begin{example}\label{Ejemplo-David}
Consider the foliation defined by the $1$-form $$ \omega=(xy+x^2y-x^2-y^2)ydx-(x-1)x^3dy$$ and its pull-back by one blow-up
 \begin{align*}
\omega^1=\frac{(x,tx)^*\omega}{x^3}&=-t \left(t -1\right) \left(t -x \right) d x -\left(x -1\right) x\, d t,\\
\omega^2=\frac{(sy,y)^*\omega}{y^3}&=
-s \left(s -1\right) \left(s^{2} y -1\right) dy+y \left(s^{2} y -s^{2}+s -1\right) ds.
\end{align*}
Observe that $\omega^1$ has a saddle-node $q$ at $(t,x)=(0,0)$  whose strong separatrix $\{t=0\}$ is transverse to the exceptional divisor $E_1=\{x=0\}$ and a radial singularity at $(t,x)=(1,0)$. On the other hand, $\omega^2$ has a saddle at $(s,y)=(0,0)$. After blowing up the singular point $(t,x)=(1,0)$ we obtain a dicritical component $E_2$ without tangencies. Thus this is a CND foliation which is not of second type.
\end{example}

\begin{example}\label{Dulac}
The Dulac-resonant foliation defined by $\omega= ydx -(2x+y^2)dy$ can be reduced with two blow-ups and have one non-degerate singularity and one corner saddle-node. Thus this is a non-CND foliation.
\end{example}

Recall that the \textbf{Milnor number} $\mu_0(\F)$ of the foliation $\F$ at $0\in\C^2$ given by the $1$-form $\omega=P(x,y)dx+Q(x,y)dy$ is defined by 
\[ \mu_0(\F)=i_0(P,Q),\] 
where $i_0(P,Q)$ denotes the intersection number of two germs $P$ and $Q$ at the origin. Remember that we consider  $P$ and $Q$  coprime, so $\mu_0(\F)$ is a non negative integer. In \cite[Theorem A]{CLS} it was proved that the Milnor number of a foliation is  a topological invariant. For instance, the Milnor number of a non-degenerate reduced singularity is $1$, whereas the Milnor number of the saddle-node  $x^k dy -y (1+\lambda x^{k-1})dx$ is $k$. 

Let $C$ be a (maybe formal) separatrix of $\F$ with primitive pa\-ra\-me\-tri\-za\-tion $\gamma: (\C,0)\to(\C^2,0)$ and $\upsilon$ a vector field defining $\F$, Camacho-Lins Neto-Sad \cite[Section 4]{CLS} defined the \textbf{multiplicity of $\F$ along $C$ at $0$} as
$\mu_0(\F,C):=\ord_t \theta(t)$,
where $\theta(t)$ is the unique vector field at $(\C,0)$ such that 
$\gamma_{*} \theta(t)=\upsilon\circ\gamma(t)$. 
If $\omega=P(x,y)dx+Q(x,y)dy$ is a 1-form inducing $\F$ and $\gamma(t)=(x(t),y(t))$, we have
\begin{equation*}
\theta(t)=
\begin{cases}
-\frac{Q(\gamma(t))}{x'(t)} & \text{if $x'(t)\neq 0$}
\medskip \\
 \frac{P(\gamma(t))}{y'(t)} & \text{if $y'(t)\neq 0$}.
\end{cases}
\end{equation*}

Following \cite[Section 2]{FP-GB-SM2024}, we define the multiplicity of  $\mathcal{F}$ along any nonempty divisor of separatrices 
$\mathcal{B}=\sum_{C} a_C\cdot C$ of separatrices of $\mathcal F$ at $0$ as follows:
\begin{equation}\label{eq:gmul}
\mu_0(\mathcal{F},\mathcal{B})=\left(\displaystyle\sum_{C}a_{C}\cdot \mu_0(\mathcal{F},C)\right)-\sum_C a_C +1.
\end{equation}
Note that this is equivalent to extend linearly the function $C \mapsto \mu_0(\mathcal{F},C)-1$. 

\begin{remark}
Recall that for the saddle-node foliation $\omega = x^k\, dy - y (1+\lambda x^{k-1})\,dx$, the strong (respectively, weak) separatrix is $S=\{x=0\}$ (respectively, $W=\{y=0\}$). It is straightforward to verify that $\mu_0(\F, S)=1$ and $\mu_0(\F, W)=k>1$.
\end{remark}

We will also need the notion of weights associated with a sequence of blow-ups
\[
\pi=\pi_1 \circ \cdots \circ \pi_n: (M,E) \to (\C^2,0),
\]
whose centers are $p_0=0, p_1, \ldots, p_{n-1}$ and whose exceptional divisor has components $E_1,\ldots,E_n$, as defined in \cite{CLS}. Recall that the weights are defined inductively by $\rho_{E_1}=1$. If $E_k=\pi_k^{-1}(p_{k-1})$ with $p_{k-1}=E_{i_1}\cap E_{i_2}$, where $i_1,i_2<k$ (respectively, $p_{k-1}\in E_i$), then $\rho_{E_k}=\rho_{E_{i_1}}+\rho_{E_{i_2}}
\quad (\text{respectively, } \rho_{E_k}= \rho_{E_i})$. It is well known that $\rho_{E_i}=\nu_0(C_i)$ (algebraic multiplicity), where $C_i$ is irreducible and its strict transform $\bar C_i$ is transverse to $E_i$ at a point that is not a corner. Hence we obtain the \textbf{vector of weights}
\[
\rho_E = (\rho_{E_1}, \ldots, \rho_{E_n})^\mathsf{T}.
\]

Following \cite{Genzmer, G-M}, we define the \textbf{tangency excess} of $\F$ by
\[
\tau_0(\F):=\sum_{q\in TSN(\bar{\F})}\sum_{E_i\ni q}\rho_{E_i}\big(\mu_q(\bar{\F},E_i)-1\big),
\]
and the \textbf{tangency excess vector}
\[
\tau_\F=(\tau_0(\F), \tau_{p_1}(\bar{\F}_1),\ldots, \tau_{p_{n-1}}(\bar{\F}_{n-1}))^\mathsf{T}.
\]

\begin{remark}
It is easy to verify that if $p$ is a reduced singularity then $\tau_p(\F)=0$.
On the other hand, although in \cite{Genzmer} the definition of $\tau_0(\F)$ is required that $\pi$ is the minimal resolution of singularities of $\F$,
 Remark~\ref{BSN} implies it takes the same value for any reduction of singularities of $\F$. 
\end{remark}

For each singular point $q$ of $\bar\F$ which is not a corner there exists a unique (formal) separatrix $B_q$ of $\F$ such that $\bar B_q$ is a separatrix of $\bar\F$ through $q$. All these separatrices of $\F$ are called isolated with respect to $\pi$.

A (formal) germ divisor $\displaystyle \B=\sum\limits_{B\in \operatorname{Sep}(\mathcal{F})} a_{B}\, B$, $a_B\in\{-1,0,1\}$ is called \textbf{balanced} adapted to $\pi$ if it satisfies the following conditions:
\begin{enumerate}[(a)]
\item if $B$ is an isolated\footnote{formal weak separatrices of saddle-nodes must be taken as isolated separatrices} separatrix of $\F$ with respect to $\pi$ then $a_B=1$, 
\item for each non-invariant (dicritical) component $D$ of the exceptional divisor $E$ of the reduction $\pi$ of singularities of $\F$ we have $\sum_{B}a_B \overline{B}\cdot D=2-\mr{val}_D$, where $\mr{val}_D$ is the valence of $D$, i.e. the number of irreducible components of $\overline{E\setminus D}$ meeting $D$, and $\overline{B}$ is the strict transform of $B$.
\end{enumerate}

Notice that a non-dicritical foliation has a unique balanced divisor which is the sum of the isolated separatrix. In contrast, a dicritical foliation have infinitely many balanced divisors. However, we can always take a balanced divisor of the form $\B =\B_0 - \B_{\infty}$, with $\B_0, \,\, \B_{\infty}\geq 0$ such that: 

\begin{enumerate}
\item $ \B_0$ is the sum of the isolated separatrices and, for each dicritical component $D$ with valence smaller than $2$, there are $2 - \mr{val}_D$ curves of the pencil of $D$. 
\item $ \B_{\infty}$ is the sum of $\mr{val}_D - 2$ curves of the pencil of each dicritical component $D$ with valence bigger than $3$.
\end{enumerate}
We say that this $\B$ is a \textbf{minimal} balanced divisor.

\section{Discrepancy and Milnor number}

Let $\F$ be a singular holomorphic foliation on $(\C^2,0)$ and let $\pi = \pi_1 \circ \cdots \circ \pi_n : (M, E) \longrightarrow (\mathbb{C}^2, 0)$  be a composition of $n$ blow-ups at points $p_0=0, p_1, \ldots, p_{n-1}$. In order to stablish our main results of this section we need to define some combinatorial data associated to $\F$ with respect to $\pi$.  
Denote by $E_1, \dots, E_n$ the irreducible components of the exceptional divisor $E = \pi^{-1}(0)$, and let $A = (A_{ij})$ be the self-intersection matrix of $E$, where 
\[
A_{ij} = E_i \cdot E_j, \qquad i,j = 1, \dots, n.
\]
Moreover, associated with $\pi$ we define a sequence of matrices $A_1, A_2, \dots, A_n = A$, where, for each $j$, the matrix $A_j$ denotes the self-intersection matrix of the exceptional divisor of $\pi_1 \circ \cdots \circ \pi_j$. We also consider the following column vectors: for any divisor $\B$
\[
S_{\mathcal{B}} = 
\begin{pmatrix}
\overline{\mathcal{B}} \cdot E_1 \\
\vdots \\
\overline{\mathcal{B}} \cdot E_n
\end{pmatrix}, 
\qquad
u = 
\begin{pmatrix}
1 \\
\vdots \\
1
\end{pmatrix},
\qquad
\iota = 
\begin{pmatrix}
\iota_1 \\
\vdots \\
\iota_n
\end{pmatrix},
\]
where $\overline{\mathcal{B}}$ denotes the strict transform of $\mathcal{B}$ by $\pi$, and 
\[
\iota_j = 
\begin{cases}
1, & \text{if } E_j \text{ is } \pi^* \mathcal{F}\text{-invariant}, \\
0, & \text{otherwise}.
\end{cases}
\] 
and let  $\delta = u-\iota$ be the vector corresponding to the dicritical components. We define a sequence of matrices $F_1, F_2, \dots, F_n$ associated with the sequence of blow-ups.  
Each matrix $F_k$ is a $k \times k$ lower triangular matrix with $1$'s along the diagonal.  
We start with $F_1 = (1),$ and for $k \geq 2$, we define
\[
F_{k} =
\begin{pmatrix}
F_{k-1} & 0 \\
- e_k & 1
\end{pmatrix},
\]
where the row vector $e_k = (e_{k,1}, \dots, e_{k,k-1})$ is given by
\[
e_{k,j} =
\begin{cases}
1, & \text{if the $k$-th blow-up point lies on the divisor } E_j, \\
0, & \text{otherwise}.
\end{cases}
\]
Denote by $F=F_n$, called the \textbf{proximity matrix}, cf. \cite[Definition 1.1.28]{Alberich} and \cite[\S3.3]{Casas}. 
The following relation holds (see \cite[Lemma 1.1.35]{Alberich} or  \cite[Lemma 2.1]{FFMR})
\begin{equation}\label{-A=F^TF}
A = - F^{\mathsf{T}} F.
\end{equation}
Also, for any curve $C$, we introduce the \textbf{vector of algebraic multiplicities} associated with the strict transforms of $C$:
$$
\nu(C) = (\nu_0(C), \nu_{p_1}(\bar C_1),\ldots, \nu_{p_{n-1}}(\bar C_{n-1}))^{\mathsf{T}}.
$$
We then have the following lemma (see \cite[Lemma 3.1]{FFMR}).
\begin{lemma}\label{multiplicidades de curva}
The following equality holds: $\nu(C)=(F^{-1})^{\mathsf{T}}S_C.$
\end{lemma}
We extend linearly $\nu$ to arbitrary divisors.
As a consequence we recover Max Noether's formula, cf. \cite[Theorem~3.3.1]{Casas}.
\begin{corollary}
If $\pi$ is a desingularization of the union $C\cup D$ of two germs of curves at $0$ without commom components, then $i_0(C,D)=\langle\nu_C,\nu_D\rangle$.
In particular, $i_0(C,D)\ge \nu_0(C)\nu_0(D)$, with equality if and only if their tangent cones  at $0$ are disjoint.
\end{corollary}

\begin{proof}
We know from \cite[Corollary 2.11]{FFMR} that  $i_0(C,D)=\langle -A^{-1}S_C,S_D\rangle$. We just need to use relation (\ref{-A=F^TF}) and previous Lemma to obtain the corollary.
\end{proof}
Another consequence is the following relation between the proximity matrix and the weights vector introduced in \cite{CLS}.

\begin{corollary}\label{weights}
The vector of weights $\rho_E=(\rho_{E_1},\ldots,\rho_{E_n})^\mathsf{T}$ coincides with
the first column of~$F^{-1}$.
\end{corollary}
\begin{proof}
Recall that $\rho_{E_i}=\nu_0(C_i)$, where $C_i$ is irreducible and $\bar C_i$ is transverse to $E_i$ at a point that is not a corner. 
By Lemma~\ref{multiplicidades de curva}, 
$$ 
\langle \rho_E, \varepsilon_i \rangle = \rho_{E_i}= \langle (F^{-1})^\mathsf{T}\varepsilon_i , \varepsilon_1\rangle=\langle (F^{-1})^\mathsf{T}\varepsilon_1 , \varepsilon_i\rangle
$$
for all $i=1, \ldots, n$, where $\varepsilon_1,\ldots,\varepsilon_n$ is the canonical basis of $\Z^n$.
\end{proof}

Let us define  the vector  $\ell=(\ell_1,\ldots,\ell_n)^{\mathsf{T}}$  of \textbf{discrepancies} of $\pi^*\F$ along each component of $E$ as follows. If $\pi_j$ is the blow up of a point $p_{j-1}$  and $\omega_j$ is a $1$-form defining the foliation around $p_{j-1}$ then $\ell_j$ is the vanishing order of $\pi_j^*(\omega_j)$ along $E_j$. In fact, it is easy to see that $\ell_j=\nu_{p_{j-1}}((\pi_{j-1}\circ\cdots\circ\pi_1)^*\F)+1-\iota_j$. This is the vector of discrepancies of the conormal bundle of $\F$ in the sense that we have the following relationship (see \cite{Brunella})
\[
N^*_{\bar \F}= \pi^*(N^*_{\F}) + \sum_{i=1}^n \ell_i E_i.
\]

\begin{remark}
One of the main results of \cite{FFMR}, Theorem 2.6, establishes the following relation for a second type foliation:
$$\ell=(F^{-1})^{\mathsf{T}}S_{\B}-F\iota. $$
A natural question is whether the converse holds, namely, whether this equality implies that the foliation is of second type. Another related question is to determine the corresponding relation for an arbitrary foliation.
Example~\ref{Ejemplo-David}  shows that the above formula does not hold in general for foliations that are not of second type.
\end{remark}

\begin{example}
A balanced divisor for the foliation of Example \ref{Ejemplo-David} defined by the $1$-form $$ \omega=(xy+x^2y-x^2-y^2)ydx-(x-1)x^3dy$$ is 
$$
\B = \{x=0\}+\{y=0\}+\{y=x\},
$$
thus we have
\[F=\left(\begin{array}{rr}1 & 0\\ -1 & 1\end{array}\right),\quad S_{\mc B}=\left(\begin{array}{c}2\\ 1\end{array}\right),\quad \iota=\left(\begin{array}{c}1\\ 0\end{array}\right),\quad \ell=\left(\begin{array}{c}3\\ 2\end{array}\right)\neq \left(\begin{array}{c}2\\ 2\end{array}\right)=(F^{-1})^\mathsf{T}S_{\mc B}-F\iota.\]
\end{example}

To obtain a general formula for the discrepancy vector, we need to introduce additional data associated with the foliation and its reduction of singularities.

Recall that we can write the balanced divisor $\mc B=\mc I+\mc D$ as the sum of all isolated separatrices -in $\mc I$-, including $B_q$ for $q\in SN(\bar\F)\setminus CSN(\bar\F)$,  and some dicritical separatrices -in $\mc D$- with coefficients $\pm1$. We can also write 
\[
\B= \left( \sum_{q \in SN \setminus TSN} B_q  + \sum_{p \in TSN \setminus CSN} B_p +\underbrace{\sum C}_{\substack{\text{isolated non-}\\ \text{degenerate separatrices}}}  \right) + \mc D
\]
and \textbf{weighted balanced} $\mc B'$ divisor of  separatrices for $\F$, requiring also the balanced con\-di\-tion on the dicritical separatrices, so that
\[
\begin{aligned}
{\mc B'} &= \left( \sum_{q \in SN \setminus TSN} B_q  
+ \sum_{p \in TSN \setminus CSN}\mu_p(\bar{\F}, E_{i_q}) B_p 
+ \sum C \right) + \mc D \\
&= {\mc B}+\sum_{q\in TSN(\bar\F)\setminus CSN(\bar\F)}(\mu_q(\bar\F, E_{i_q})-1)B_q
\end{aligned}
\]
We define the \textbf{tangency saddle-node vector} of $\F$ by means
\[T_\F=\left(\sum\limits_{q\in TSN(\bar\F)\cap E_1}(\mu_q(\bar\F,E_1)-1),\ldots,\sum\limits_{q\in TSN(\bar\F)\cap E_n}(\mu_q(\bar\F,E_n)-1)\right)^\mathsf{T}\]
and we put
\[S_{\F}:=S_{\mc B}+T_\F=S_{\mc B'}+C_\F,\]
where
\[C_\F=\left(\sum\limits_{q\in CSN(\bar\F)\cap E_1}(\mu_q(\bar\F,E_1)-1),\ldots,\sum\limits_{q\in CSN(\bar\F)\cap E_n}(\mu_q(\bar\F,E_n)-1)\right)^\mathsf{T}.\]
\begin{remark}
Notice that 
\begin{enumerate}
\item Both vectors $T_{\F}$ and $C_{\F}$ have non-negative entries.
\item Since for a non-tangent saddle node $q \in E_j$ we have $\mu_q(\bar\F,E_j)=1$, we have 
\[T_\F=\left(\sum\limits_{q\in SN(\bar\F)\cap E_1}(\mu_q(\bar\F,E_1)-1),\ldots,\sum\limits_{q\in SN(\bar\F)\cap E_n}(\mu_q(\bar\F,E_n)-1)\right)^\mathsf{T}.\]
\item
$\F$ is of second type if and only if $T_\F=0$ (i.e. $S_\F=S_{\mc B}$),
\item
 $\F$ is CND if and only if $C_\F=0$ (i.e. $S_\F=S_{\mc B'}$).
\item It is clear that $\langle \rho_E, T_{\F}\rangle = \tau_0(\F)$.
 \end{enumerate}
\end{remark} 
 
\begin{remark}\label{F}
 If $\pi_i$ is the blow-up of a point $p_i\in(\pi_{1}\circ\cdots\circ\pi_{i-1})^{-1}(p_0)$ for $i=2,\ldots,n$ then
\[F(\pi_1\circ\cdots\circ\pi_n)=\begin{pmatrix} 1 & 0\\ f_n & F(\pi_2\circ\cdots\circ\pi_n)\end{pmatrix} \Rightarrow
F(\pi_1\circ\cdots\circ\pi_n)^{-1}=\begin{pmatrix} 1 & 0\\ g_n & F(\pi_2\circ\cdots\circ\pi_n)^{-1}\end{pmatrix}.\]
The first column $(1,g_n)^\mathsf{T}$ of $(F(\pi_n\circ\cdots\circ\pi_1))^{-1}$ is the vector of weights $\rho(\pi_1\circ\cdots\circ\pi_n)=(\rho_{E_1},\ldots,\rho_{E_n})$ associated to $p_0$. Then the second column of $(F(\pi_1\circ\cdots\circ\pi_n))^{-1}$ is of the form $(0,\rho(\pi_2\circ\cdots\circ\pi_n))^\mathsf{T}$ associated to $p_1$ and so on. 
\end{remark}
 
\begin{proposition}\label{tau y T_F}
For a reduction of singularities $\pi:(M,E)\to (\C^2,0)$ of $\F$ we have  
\[(F^{-1})^\mathsf{T}T_\F=\tau_\F.\]
\end{proposition}
 
\begin{proof}
For the first component of the vectors we have: $\tau_0(\F)= \langle \rho_E, T_{\F} \rangle = \langle (F^{-1})^\mathsf{T}T_\F , \varepsilon_1\rangle $, where the last equality follows from Corollary \ref{weights}. We verify now the equality in the second component. If $T_\F(\pi_1\circ\cdots\circ\pi_n)=(t_1,\ldots,t_n)^\mathsf{T}$ then clearly $T_{\bar\F_1}(\pi_2\circ\cdots\circ\pi_n)=(t_2,\ldots,t_n)^\mathsf{T}$. Therefore 
$$
\langle (F^{-1})^\mathsf{T}T_\F , \varepsilon_2\rangle= \langle (0,\rho(\pi_2\circ\cdots\circ\pi_n))^\mathsf{T}, T_{\F} \rangle = \tau_{p_1}(\bar\F_1),
$$
where the first equality follows from previous remark. The remaining components are treated analogously.
\end{proof} 

We denote by $\nu_0(\F)$ the algebraic multiplicity of $\F$ at $0$.
In \cite{Genzmer} second type foliations are characterized (see also \cite[Proposition 3.3]{G-M}):

\begin{proposition}[\cite{Genzmer}, Proposition 2.4]\label{Genzmer}
Let $\B$ a balanced divisor for $\F$, then $$\nu_0(\F)=\nu_0(\B)-1 + \tau_0(\F).$$
In particular $\F$ is of second type if and only if $\nu_0(\F)=\nu_0(\B)-1$.
\end{proposition}

Now we are ready to state our first main result.

 \begin{theorem}\label{ell}
Let $\F$ be an arbitrary foliation on $(\C^2,0)$, let $\pi$  be a reduction of singularities of $\F$ and let  $\mc B$  be a balanced divisor adapted to $\pi$. Then we have
 \[\ell=(F^{-1})^\mathsf{T}S_\F-F\iota = \nu_{\B}+\tau_\F -F\iota. \]
In particular $\ell=(F^{-1})^\mathsf{T}S_{\mc B}-F\iota$ if and only if $\F$ is of second type.
\end{theorem}

\begin{proof}
The second equality follows from Proposition \ref{tau y T_F} together with the fact that $(F^{-1})^\mathsf{T} S_{\B}=\nu(\B)$ by Lemma \ref{multiplicidades de curva}. We establish the first equality by induction on the number $n$ of blow-ups in the reduction $\pi$ of $\F$.

If $n=1$ we have, by Genzmer's proposition, $\ell_1 = \nu_0(\F)+1-\iota_1 =\nu_0(\B) + \tau_0(\F)-\iota_1$. 
Assume now that the result holds for $n-1$ and let us prove it for $n$. 
Let $\pi_0$ be the blow-up of the origin and consider the set $\{p_1,\ldots,p_r\}\subset E_1=\pi_0^{-1}(0)$ of singular points of $\pi_0^*\F$ that we must blow up to obtain the  given reduction $\pi=\pi_0\circ\cdots\circ \pi_{r-1}\circ\pi_r$  of singularities of $\F$. Notice that $\pi_i$ is a composition of blow-ups starting at $p_i$, $i=1,\ldots,r$.
As before we denote $\bar\F=\pi^*\F$. Let us denote by $\F_i$ the germ of $\pi_0^*\F$ at $p_i$. We decompose the balanced divisor $\mc B=\sum_{i=0}^r\mc B_i$ so that the strict transform of each component of $\mc B_i$ by $\pi_0$ passes through the point $p_i$ for $i=1,\ldots,r$ and the strict transform of each component of $\mc B_0$ does not pass through any $p_i$.
With the obvious notations, notice that the vectors $\ell_\F$, $T_\F$, $\tau_\F$, $\iota_\F$  and $S_{\mc B}$ associated to $\F$ and $\pi$ can be written as
\begin{align*}
\ell_\F&=(\nu_0(\F)+1-\iota_1,\ell_{\F_1},\ldots,\ell_{\F_r})^\mathsf{T},\\
T_\F&=\left(\sum_{q\in TSN(\bar\F)\cap E_1}(\mu_q(\bar\F,E_1)-1),T_{\F_1},\ldots,T_{\F_r}\right)^\mathsf{T},\\
\tau_\F&=(\tau_0(\F),\tau_{\F_1},\ldots,\tau_{\F_r})^\mathsf{T},\\
\iota_\F&=(\iota_1,\iota_{\F_1},\ldots,\iota_{\F_r})^\mathsf{T},\\
S_{\mc B}&=(\bar{\mc B}\cdot E_1,S_{\mc B_1},\ldots,S_{\mc B_r})^\mathsf{T}.
\end{align*}
On the other hand, the proximity matrix  $F$ of $\pi$ and its inverse $F^{-1}$ take the form
\[F=\left(\begin{array}{cccc} 1 & 0 & \cdots & 0\\ 
f_1 & F_1 & \cdots & 0\\
\vdots & \vdots & \ddots & \vdots\\
f_r & 0 & \cdots & F_r
\end{array}\right),\qquad 
F^{-1}=\left(\begin{array}{cccc} 1 & 0 & \cdots & 0\\ 
\rho_1 & F_1^{-1} & \cdots & 0\\
\vdots & \vdots & \ddots & \vdots\\
\rho_r & 0 & \cdots & F_r^{-1}
\end{array}\right),
\]
where the first column of $F^{-1}$ is the vector of weights associated to $\pi$. In fact, each vector $\rho_i$ is the vector of weights associated to the composition $\pi_i$ of blow-ups in $\pi$ reducing the singularity $p_i$ of $\pi_0^*\F$. 
Notice that if $\bar{\mc B}_i$ is the strict transform of $\mc B_i$ by $\pi_0$ then
 $\bar{\mc B}_i+\iota_1E_1$ is a balanced divisor for $\F_i$ adapted to $\pi_i$ for $i=1,\ldots,r$.
By the inductive hypothesis we have
$\ell_{\F_i}=(F_i^{-1})^\mathsf{T}(S_{\mc B_i}+\iota_1 S^i_{E_1}+T_{\F_i})-F_i\iota_{\F_i}$,
where $S^i_{E_1}=(E_1\cdot E_{i1},\ldots,E_1\cdot E_{in_i})^\mathsf{T}$ and $\pi_i^{-1}(p_i)=\bigcup_{j=1}^{n_i}E_{ij}$.
We have that
\begin{align*}
(F^{-1})^\mathsf{T}S_\F-F\iota_\F=&\left(\begin{array}{cccc} 1 & \rho_1^\mathsf{T} & \cdots & \rho_r^\mathsf{T}\\ 
0 & (F_1^{-1})^\mathsf{T} & \cdots & 0\\
\vdots & \vdots & \ddots & \vdots\\
0 & 0 & \cdots & (F_r^{-1})^\mathsf{T}
\end{array}\right)\left(\begin{array}{c}
\bar{\mc B}\cdot E_1+\sum\limits_{q\in TSN(\bar\F)\cap E_1}(\mu_q(\bar\F,E_1)-1)\\
S_{\mc B_1}+T_{\F_1}\\
\vdots\\
S_{\mc B_r}+T_{\F_r}
\end{array}\right)\\
& -\left(\begin{array}{cccc} 1 & 0 & \cdots & 0\\ 
f_1 & F_1 & \cdots & 0\\
\vdots & \vdots & \ddots & \vdots\\
f_r & 0 & \cdots & F_r
\end{array}\right)\left(\begin{array}{c}\iota_1\\ \iota_{\F_1}\\ \vdots\\ \iota_{\F_r}\end{array}\right)\\
=&\left(\begin{array}{c}\bar{\mc B}\cdot E_1+\sum\limits_{i=1}^r\rho_i^\mathsf{T}S_{\mc B_i}+\tau_0(\F)-\iota_1\\
(F_1^{-1})^\mathsf{T}(S_{\mc B_1}+T_{\F_1})-\iota_1f_1-F_1\iota_{\F_1}\\
\vdots\\
(F_r^{-1})^\mathsf{T}(S_{\mc B_r}+T_{\F_r})-\iota_1f_r-F_r\iota_{\F_r}
\end{array}\right)\stackrel{(\star)}{=}\left(\begin{array}{c}
\nu_0(\mc B)+\tau_0(\F)-\iota_1\\
\ell_{\F_1}\\
\vdots\\
\ell_{\F_r}
\end{array}\right)=\ell_\F,
\end{align*}
where in the  second equality we note that $\sum_{q\in TSN\cap E_1}(\mu_q(\bar\F)-1)+\sum_{i=1}^r \rho_i^\mathsf{T}T_{\F_i}=\tau_0(\F)$,  in the last equality we use Proposition~\ref{Genzmer}, and in equality  ($\star$) 
we note that $\bar{\mathcal B}\cdot E_1+\sum_{i=1}^r\rho_i^\mathsf{T}S_{\mathcal B_i}=\nu_0(\mathcal B)$ for the first component and use that $S_{E_1}^i=-F_i^\mathsf{T}f_i$ for $i=1,\ldots,r$ for the remaining components.This fact can be checked as follows:
assume that after blowing-up the point $p_i$ the next centers are $E_1\cap E_{11}, E_1\cap E_{12},\ldots, E_1\cap E_{1k}$ with $1\le k\le n_i$, and the remaining centers are not longer on the strict transform of $E_1$ then the unique non-zero component of $S_{E_1}^i$ is a $1$ in the position $k$.
On the other hand,
\[\left(\begin{array}{c|c}1 & 0\\ \hline  f_i & F_i\end{array}\right)=\left(\begin{array}{c|cc}
1 & 0 & 0\\ \hline -u & G_1 & 0\\ 0 & G_2 & G_3
\end{array}\right),\]
where $u=(-1,\ldots,-1)^\mathsf{T}$ is a vector of $k$ components, $G_1$ is the $k\times k$ matrix
\[G_1=\left(\begin{array}{ccccc}
1 & 0 & 0 &  \cdots & 0\\
-1 & 1 & 0 & \cdots & 0\\
0 & -1 & 1 & \cdots & 0\\
\vdots & \vdots  & \ddots & \ddots & \vdots\\
0 & 0 & \cdots & -1 & 1
\end{array}\right)\]
and $G_3$ is a lower triangular square matrix.
Then the unique non-zero component of the vector $-F_i^\mathsf{T}f_i=\left(\begin{array}{cc}
G_1^\mathsf{T} & G_2^\mathsf{T}\\ 0 & G_3^\mathsf{T}\end{array}\right)\left(\begin{array}{c}u\\ 0
\end{array}\right)
=\left(\begin{array}{c}G_1^\mathsf{T} u\\ 0\end{array}\right)$ is a $1$ in the position $k$. Hence $-F_i^\mathsf{T}f_i=S_{E_1}^i$.
\end{proof}

\begin{example}
For the foliation of Example \ref{Ejemplo-David}, 
$$
\F : \,\, \omega = (xy+x^2y-x^2-y^2)ydx-(x-1)x^3dy
$$
we have seen that the reduction of singularities has a tangent saddle-noe $q$ at the first divisor, thus we can easily see that 
$$
T_\F=\left(\begin{array}{c} \mu_q(\bar\F,E_1)-1\\ 0\end{array}\right)=\left(\begin{array}{c}1\\ 0\end{array}\right).
$$
Therefore 
$$
(F^{-1})^\mathsf{T}S_{\mc F}-F\iota =\left(\begin{array}{c}3\\ 2\end{array}\right)= \ell.
$$
\end{example}

We consider now the vector of algebraic multiplicities
\[
\nu_\F := (\nu_0(\F),\nu_{p_1}(\bar{\F}_{p_1}),\ldots,\nu_{p_{n-1}}(\bar{\F}_{p_{n-1}}))^\mathsf{T}.
\]
Since $\ell_j=\nu_{p_j}(\bar{\F}_j)+1-\iota_j,$ it follows that
\[
\ell=\nu_\F+u-\iota.
\]
Hence we obtain the following corollary.

\begin{corollary}
For a foliation $\F$ on $(\C^2,0)$ one has
\[
\nu_{\F}=  (F^{-1})^\mathsf{T}S_\F - F\iota - \delta.
\]
\end{corollary}

As a consequence of the previous theorem we obtain a formula for the Milnor number from the reduction data, generalizing \cite[Theorem 2.8]{FFMR}.

\begin{remark}
Notice that the Milnor number of a saddle-node $x^kdy+(\lambda y+\cdots)dx$ is $k$.
\end{remark}
Before state our second theorem, we need to define, for any $\F$-invariant curve $C$, the \textbf{transverse excess} of $\F$ over $C$ by
\[\mathfrak{T}_\F(C)=\sum_{p\in \bar C\cap SN(\bar\F)\setminus TSN(\bar\F)}(\mu_p(\bar\F)-1).\]
We extend by linearity $\mathfrak{T}_\F$ for arbitrary divisors of separatrices. Notice that for a balanced divisor $\B$, we have $\mathfrak{T}_\F(\mc B)=\mathfrak{T}_\F(\mc I)$.

\begin{theorem}\label{muF}
 Let $\F$ be an arbitrary foliation on $(\C^2,0)$, let $\pi$ be a reduction of singularities of $\F$ and let $\mc B$ be a balanced divisor adapted to $\pi$. Then
\[
\mu_0(\F)=\langle-A^{-1}S_\F -(I+F^{-1})u   ,S_\F\rangle
+1+
\mathfrak{T}_\F(\mc B).
\]
In particular
\begin{enumerate}
\item If $\F$ is of second type, then $$\mu_0(\F)=\langle-A^{-1}S_\B -(I+F^{-1})u   ,S_\B\rangle
+1+ \sum\limits_{p\in SN(\bar\F)}(\mu_p(\bar\F)-1).$$
\item If $\F$ is a generalized curve, then  $$\mu_0(\F)=\langle-A^{-1}S_\B -(I+F^{-1})u   ,S_\B\rangle
+1.$$
\end{enumerate}
\end{theorem}

\begin{proof}
The Van den Essen formula implies that 
\[\mu_0(\F)=\sharp\mr{Sing}(\bar\F)+N(\ell) +C,\] 
where $N(\ell)=\sum_{j=1}^n(\ell_j^2-\ell_j-1)=\langle\ell,\ell\rangle-\langle\ell,u\rangle-\langle u,u\rangle$ and $C=\sum_{p\in {SN}(\bar\F)}(\mu_p(\bar\F)-1)$.
We can write $\B=\mc I+\mc D$ where $\mc I$ are the isolated separatrices  (including the formal ones) of $\F$ and the support of $\mc D$ consists of some dicritical separatrices of $\F$. The number of singularities of $\bar\F$ is 
\[ 
\underbrace{
\langle S_{\mc I},u\rangle
}_{\begin{array}{c}\scriptstyle\text{attaching points of}\\ \scriptstyle\text{isolated separatrices}\end{array}}
+\underbrace{\langle u,u\rangle-1}_{\text{corners}}-\underbrace{\sum_{D\text{ dic}}\mr{val}_D}_{\begin{array}{c}\scriptstyle\text{attaching points of}\\ \scriptstyle\text{dicritical components}\end{array}}\] 
and, using that $\mc B$ is balanced, we have $\sum_{D\text{ dic}}\mr{val}_D=2\langle\delta,u\rangle-\langle S_{\mc B},\delta\rangle$, where $\delta=u-\iota$ is the vector of dicritical components of~$E$.
Since  $S_{\mc B}=S_{\mc I}+S_{\mc D}$ and $\langle S_{\mc I},\delta\rangle=\langle S_{\mc D},\iota\rangle=0$ we deduce that $\langle S_{\mc I},u\rangle=\langle S_{\mc B},u-\delta\rangle$.
Thus we obtain 
\begin{align*}
\mu_0(\F)&=\langle S_{\mc B},u-\delta\rangle+\langle u,u\rangle-1-\Big(2\langle\delta,u\rangle-\langle S_{\mc B},\delta\rangle\Big)+\langle\ell,\ell\rangle-\langle \ell,u\rangle-\langle u,u\rangle+C\\
&=\langle S_{\mc B},u\rangle-2\langle \delta,u\rangle+\langle\ell,\ell\rangle-\langle\ell,u\rangle-1+C.
\end{align*}
Since, according to Theorem~\ref{ell},

\begin{align*}
\langle\ell,\ell\rangle&=\langle (F^{-1})^{\mathsf{T}}S_{\F},(F^{-1})^{\mathsf{T}}S_{\F}\rangle-2\langle (F^{-1})^{\mathsf{T}}S_{\F},F\iota\rangle+\langle F\iota,F\iota\rangle\\
&=\langle -A^{-1}S_{\F},S_{\F}\rangle-2\langle S_{\F},\iota\rangle-\langle \iota,A\iota\rangle
\end{align*}
and
\[\langle \ell,u\rangle=\langle (F^{-1})^{\mathsf{T}}\big(S_{\mc F}-F^{\mathsf{T}}F\iota\big),u\rangle=\langle S_{\F} + A\iota,F^{-1}u\rangle,\]
we deduce that
\begin{align*}
\mu_0(\F)&=\langle S_{\mc B},u\rangle-2\langle \delta,u\rangle+\langle -A^{-1}S_{\F},S_{\F}\rangle-2\langle S_{\F},u-\delta\rangle-\langle \iota, A\iota\rangle-\langle S_{\F}+A\iota,F^{-1}u\rangle-1+C.\\
\end{align*}

Observe that 
\begin{align*}
\langle S_{\mc B},u\rangle + C &= \langle S_{\mc F},u\rangle + C-\langle T_\F,u\rangle \\
&=\langle S_{\mc F},u\rangle+\sum_{p\in SN(\bar\F)\setminus TSN(\bar\F)}(\mu_p(\bar \F)-1)\\
&=\langle S_{\mc F},u\rangle+\mathfrak{T}_\F(\mc B)
\end{align*}
because
\[\langle T_\F,u\rangle=\sum_{i=1}^n\sum_{p\in TSN(\bar\F)\cap E_i}(\mu_p(\bar \F,E_i)-1)=\sum_{p\in TSN(\bar \F)}(\mu_p(\bar\F)-1)\]
due to the fact that $\mu_p(\bar\F,E_i)=1$ (resp. $\mu_p(\bar\F,E_i)=\mu_p(\bar\F)$) if $E_i$ is the strong (resp. weak) separatrix of $p$. Therefore

\begin{align*}
\mu_0(\F)
&=\begin{aligned}[t]
   &\langle -A^{-1}S_{\F},S_{\F}\rangle+\langle S_{\F},u\rangle
   -2\langle \delta,u\rangle-2\langle S_{\F},u-\delta\rangle -\langle \iota, A\iota\rangle-\langle S_{\F}+A\iota,F^{-1}u\rangle\\
   &\quad   -1+\mathfrak{T}_\F(\mc B)
  \end{aligned}\\
&=\langle-A^{-1}S_\F -(I+F^{-1})u,S_\F\rangle
  +2\langle S_{\F},\delta\rangle
  -\langle A\iota,F^{-1}u\rangle
  -2\langle\delta,u\rangle
  -\langle\iota,A\iota\rangle
  -1 +\mathfrak{T}_\F(\mc B)\\
&=\langle-A^{-1}S_\F -(I+F^{-1})u,S_\F\rangle
  -1+\beta +\mathfrak{T}_\F(\mc B),
\end{align*}
where $\beta=2\langle S_{\mc F},\delta\rangle-\langle A\iota,F^{-1}u\rangle-2\langle\delta,u\rangle-\langle\iota,A\iota\rangle$. Notice that
\[\langle S_\F,\delta\rangle=\langle S_{\mc B},\delta\rangle+\langle T_\F,\delta\rangle=\langle S_{\mc B},\delta\rangle\]
because $TSN(\bar \F)\cap E_i=\emptyset $ if $E_i$ is dicritical. Then $\beta=2\langle S_{\mc B},\delta\rangle-\langle A\iota,F^{-1}u\rangle-2\langle\delta,u\rangle-\langle\iota,A\iota\rangle$ and the same proof of \cite[Theorem 2.8]{FFMR} shows that $\beta =2$. This completes the proof.
\end{proof}

\section{Index theorems for holomorphic foliations} 
Throughout this section
we fix a germ of holomorphic foliation $\F : \omega = 0$ on $(\C^2,0)$ and we consider reduced
 $\F$-invariant divisors $C=C_0-C_\infty$ 
 where $C_0$ and $C_\infty$ are reduced and effective. 
 Let $\pi=\pi_n\circ\cdots\circ\pi_1$ be a resolution of $\F$ with intersection matrix $A=-F^{\mathsf{T}}F$, proximity matrix $F$ and invariant vector $\iota$.
We will denote by $\mc B$ a balanced divisor for $\F$ adapted to $\pi$.  For the definition and basic properties of the indices we consider in this section we refer to \cite{Brunella}.

Consider $C=\{f=0\}$ a reduced effective divisor invariant by $\F$, then we can write (\cite{Su})
\begin{equation}\label{descomposicion}
g\omega = hdf + f \eta,
 \end{equation}
with $f$ and $h$  relatively prime and $\eta$ a holomorphic $1$--form. We define  the \textbf{G\'omez-Mont- Seade - Verjovsky index} (GSV index) of $\mathcal{F}$ with respect to $C$ as 
\begin{equation*}
GSV_0(\mathcal{F}, C)= ord_0 \left.\left( \frac{h}{g}\right)\right|_C=\frac{1}{2\pi i}\int_{\partial C}\frac{g}{h}d\left(\frac{h}{g} \right),
\end{equation*} 
where $\partial C$ is the intersection of $C$ with a small sphere around $0$, with the induced orientation. A decomposition (\ref{descomposicion}) also exists for a branch of formal separatrix $C$ with formal equation $f \in \C[[x, y]]$, yielding $g$, $h$ and $\eta$ as formal objects. In this context, we
can extend the definition of the GSV index to $C$ by taking $\gamma(t)$, a Puiseux parametrization for $C$ such that $\gamma(0) = 0$, and setting \[GSV_0(\mathcal{F}, C)= ord_0 \left.\left( \frac{h}{g} \circ \gamma(t)\right)\right..\]
If $C_1$ and $C_2$ are $\F$-invariant curves without common components, then the following formula holds (cf. \cite{Brunella})
\begin{equation}\label{GSV en suma}
GSV_0(\mathcal{F}, C_1+C_2)=GSV_0(\mathcal{F}, C_1)+GSV_0(\mathcal{F}, C_2)- 2 i_0(C_1,C_2).
\end{equation}

\begin{remark}
For the saddle-node foliation given by $\omega = x^k\, dy - y (1+\lambda x^{k-1})\,dx$ with strong (respectively, weak) separatrix $S=\{x=0\}$ (respectively, $W=\{y=0\}$) we can verify that
\begin{align*}
GSV_0(\F,W)&=k=\mu_0(\F),\\
GSV_0(\F,S)&=1.
\end{align*}
\end{remark}

\begin{remark}\label{GSV vs Milnor}
It can be seen that, for $C$ an invariant (formal) curve, $\mu_0(\F,C)=GSV_0(\mathcal{F}, C) + \mu_0(C)$. In fact, it follows easily from the definition if $C$ is irreducible, and from the behavior of the indices with respect to the sum of curves in the general case.
\end{remark}

Now, we define the \textbf{Camacho-Sad index} of $\F$ along $C$ as
$$
CS_0(\F,C) = -\frac{1}{2\pi i}\int_{\partial C}\frac{1}{h}\eta = - Res_{t=0} \gamma^*\left(\frac{1}{h}\eta\right).
$$
In contrast to the GSV-index, see \eqref{GSV en suma}, if $C_1$ and $C_2$ are $\mathcal{F}$-invariant curves without common components, the following holds (cf. \cite{Brunella}):
\begin{equation}\label{CS en suma}
CS_0(\F,C_1+C_2)=CS_0(\F,C_1)+CS_0(\F,C_2)+2i_0(C_1,C_2).
\end{equation}
We extend the definition for divisors with polar part as follows. Let $C=C_0-C_\infty$ be an invariant reduced divisor with $C_0$ and $C_\infty$ effective. We define \[CS_0(\F,C)=CS_0(\F,C_0)+CS_0(\F,C_\infty)-2i_0(C_0,C_\infty).\]
On a pointed neighborhood 
of $0$ we may find a complex valued smooth $1$-form $\beta$, of type $(1,0)$, such that
\begin{equation}\label{descomposicion beta}
d\omega = \beta \wedge \omega.
\end{equation}
The \textbf{Variation index} of $\F$ along $C$ 
is defined as (\cite{K-S})
$$
Var_0(\F,C) = \frac{1}{2\pi i}\int_{\partial C} \beta.
$$
For any invariant curve we have the relation (cf. \cite[Proposition 5]{Brunella})
$$
Var_0(\F,C)= CS_0(\F,C)+ GSV_0(\F,C).
$$
and we use this relation to define the variation index for a formal separatrix. This index is additive in the separatrices of $\F$
$$
Var_0(\F,C_1+C_2)= Var(\F,C_1)+ Var_0(\F,C_2). 
$$
So we extend $Var_0$ by linearity for arbitrary invariant reduced divisors. Finally, using the writing \eqref{descomposicion beta}, the \textbf{Baum-Bott} index of $\F$ at $0$ is 
$$
BB_0(\F)=\frac{1}{(2 \pi i)^2}\int_{S^3} \beta \wedge d\beta
$$
where $S^3$ in a small sphere around $0$, oriented as a boundary of a small ball containing $0$. For a non-degenerate reduced singularity with eigenvalues $\lambda_1$ and $\lambda_2$ we have
$$
BB_0(\F)= \frac{\lambda_1}{\lambda_2}+\frac{\lambda_2}{\lambda_1}+2.
$$

\begin{theorem}\label{Indices}
Let $\F$ be a germ of foliation on $(\C^2,0)$ and $\B$ a balanced divisor. Then:
\begin{enumerate}
\item If $C$ is an invariant divisor, then
\begin{align*}
\mu_0(\F,C)
  &= \langle -A^{-1}S_{\F}-(I+F^{-1})u,S_C\rangle 
     + 1 + \mathfrak{T}_{\F}(C).
\end{align*}

\item If, in addition, $C$ is reduced, then
\begin{align*}
CS_0(\F,C)
  &= \sum_{p\in \bar C\cap E} 
     \bigl(CS_p(\bar\F,\bar C_0)+CS_p(\bar\F,\bar C_\infty)\bigr)
     + \langle -A^{-1}S_C,S_C\rangle,\\[4pt]
Var_0(\F,C)
  &= \sum_{p\in E\cap\bar C_0} Var_p(\bar\F,\bar C_0)
     - \sum_{p\in E\cap\bar C_\infty} Var_p(\bar\F,\bar C_\infty)
     + \langle -A^{-1}S_{\F}-\iota,S_C\rangle.
\end{align*}

\item If, moreover, $C$ is reduced and effective, then
\begin{align*}
GSV_0(\F,C)
  &= \langle -A^{-1}(S_{\F}-S_C),S_C\rangle
     + \mathfrak{T}_{\F}(C).
\end{align*}

\item The Baum--Bott index is given by
\begin{align*}
BB_0(\F)
  &= \sum_{p\in\mathrm{Sing}(\bar\F)} BB_p(\bar\F)
     + \langle -A^{-1}S_{\F},S_{\F}\rangle
     - 2\langle S_{\F},\iota\rangle
     - \langle A\iota,\iota\rangle.
\end{align*}
\end{enumerate}
\end{theorem}

\begin{proof}
Let us begin proving (3) in the case that $C$ is irreducible.
By \cite{Brunella}, the relation between  the GSV index of a foliation $\F$ and its pull-back $\bar\F$ by a single blow-up is 
\[GSV_0(\F,C)=GSV_p(\bar\F,\bar C)+\nu_0(C)(\ell_1-\nu_0(C)),\]
where $\bar C$ is the strict transform of $C$ and $p$ is the intersection point of $\bar C$ with the exceptional divisor. Recursively we obtain
\[GSV_0(\F,C)=GSV_{p}(\bar\F,\bar C)+\sum_{i=1}^n\nu_{p_{i-1}}(\bar C_{i-1})(\ell_i-\nu_{p_{i-1}}(\bar C_{i-1})=GSV_{p}(\bar\F,\bar C)+\langle \nu_C,\ell-\nu_C\rangle,\]
where $\{p\}=\bar C\cap E=\bar C\cap E_k$.
Using Lemma~\ref{multiplicidades de curva} and Theorem~\ref{ell} we have
\begin{align*}
\langle\nu_C,\ell-\nu_C\rangle
&=\langle (F^{-1})^{\mathsf{T}}S_C,\ell-(F^{-1})^{\mathsf{T}}S_C\rangle\\
&=\langle (F^{-1})^{\mathsf{T}}S_C,(F^{-1})^{\mathsf{T}}S_{\F}-F\iota-(F^{-1})^{\mathsf{T}}S_C\rangle\\
&=\langle S_C,F^{-1}(F^{-1})^{\mathsf{T}}S_{\F}-\iota-F^{-1}(F^{-1})^{\mathsf{T}}S_C\rangle\\
&=\langle S_C,-A^{-1}(S_{\F}-S_C)-\iota\rangle\\
&=\langle S_C,-A^{-1}(S_{\F}-S_C)\rangle-\underbrace{\langle S_C,\iota\rangle}_{\iota_k}.
\end{align*}
On the other hand, 
\begin{equation}
\label{GSVp}
GSV_{p}(\bar \F,\bar C)=\left\{\begin{array}{ll} 0 & \text{ if $p$ is non-singular,}\\
\mu_{p}(\bar\F) &  \text{ if $p$ is a non-tangent saddle-node,}\\
1 &\text{ otherwise.}
\end{array}\right.
\end{equation}
Hence,
\[GSV_0(\F,C)=\langle S_C,-A^{-1}(S_{\F}-S_C)\rangle+\iota_k(GSV_{p}(\bar\F, \bar C)-1)=\langle S_C,-A^{-1}(S_{\F}-S_C)\rangle+\mathfrak T_\F(C).\]
The general case follows by using that
\[GSV_0(\F,C_1+C_2)=GSV_0(\F,C_1)+GSV_0(\F,C_2)-2i_0(C_1,C_2),\]
$S_{C_1+C_2}=S_{C_1}+S_{C_2}$, $i_0(C_1,C_2)=\langle S_{C_1},-A^{-1}S_{C_2}\rangle$ and $\mathfrak T_\F(C_1+C_2)=\mathfrak T_\F(C_1)+\mathfrak T_\F(C_2)$.

We prove now assertion (1) in case that $C$ is effective.  The foliation $\F_C$ defined by a reduced equation of $C$ has $S_{\F_C}=S_C$.
Using Theorem~\ref{muF} for the foliation $\F_C$  after the  first equality (see Remark~\ref{GSV vs Milnor})  we  obtain
\begin{align*}
\mu_0(\F,C)&=GSV_0(\mathcal{F}, C) + \mu_0(C) \\
&= \langle -A^{-1}(S_{\F}-S_C),S_C\rangle +\mathfrak T_\F(C)+\langle-A^{-1}S_C -(I+F^{-1})u   ,S_C\rangle
+1\\
&=\langle -A^{-1}S_\F-(I+F^{-1})u,S_C\rangle+1+\mathfrak T_\F(C).
\end{align*}
Finally, we use the fact that both $\mu_0(\F, C) - 1$, $\langle -A^{-1}S_{\mathcal{B}} - (I + F^{-1})u, S_C \rangle$ and $\mathfrak T_\F$ are linear in $C$ to state  the result for any reduced divisor with polar part.

The proof of assertion (2) for the Camacho-Sad index that we give in \cite{FFMR} for generalized curves is valid for arbitrary foliations because in the relation $CS_0(\F,C)=CS_{p_1}(\bar\F_1,\bar C_1)+\nu_0(C)^2$ does not appears $\ell_1$. For the variation index we consider first $C$ effective and use the relation (cf. \cite[Proposition~5]{Brunella})
\begin{equation}
\label{VarCS}
Var_0(\F,C)=CS_0(\F,C)+GSV_0(\F,C)
\end{equation}
and the following easy consequence of formula (\ref{GSVp})
\begin{equation}\label{sumGSV}
\sum_{p\in E\cap \bar C}GSV_p(\bar\F,\bar C\rangle=\langle\iota,S_C\rangle+\mathfrak T_\F(C).
\end{equation}

For a divisor $C$ with polar part we use the linearity in the definition of the variation. Finally, the assertion (4) concerning the Baum-Bott index follows from the relation (see \cite[Proposition~1]{Brunella})
\[BB_0(\F)=\sum_{p\in\mr{Sing}(\bar\F_1)}BB_p(\bar\F_1)+\ell_1^2.\]
Recursively we obtain that
\begin{align*}
BB_0(\F)&=\sum_{p\in\mr{Sing}(\bar\F)}BB_p(\bar\F)+\langle\ell,\ell\rangle\\
&=\sum_{p\in\mr{Sing}(\bar\F)}BB_p(\bar\F)+\langle -A^{-1}S_{\F},S_{\F}\rangle -2\langle S_{\F},\iota\rangle-\langle A\iota,\iota\rangle,
\end{align*}
where in the second equality we use Theorem~\ref{ell}.
\end{proof}

In the case of foliations of second type, we obtain the following formulas, except for the Camacho–Sad index, which remains unchanged.
\begin{corollary}\label{Arturo}
With the same notation of the previous theorem, if $\F$ is a second type foliation, we have
\setcounter{equation}{0}
\begin{align}
\mu_0(\F,C)
 &= \langle -A^{-1}S_{\mc B}-(I+F^{-1})u,S_C\rangle +1
    +\sum_{p\in {SN}(\bar\F)\cap\bar C}(\mu_p(\bar\F)-1),\\
Var_0(\F,C)
 &= \sum_{p\in E\cap\bar C_0} Var_p(\bar\F,\bar C_0)
    - \sum_{p\in E\cap\bar C_\infty} Var_p(\bar\F,\bar C_\infty)
    + \langle -A^{-1}S_{\B}-\iota,S_C\rangle,\\
GSV_0(\F,C)
 &= \langle -A^{-1}(S_{\mc B}-S_C),S_C\rangle
    +\sum_{p\in {SN}(\bar\F)\cap\bar C}(\mu_p(\bar\F)-1),\\
BB_0(\F)
 &= \sum_{p\in\mathrm{Sing}(\bar\F)} BB_p(\bar\F)
    + \langle -A^{-1}S_{\B},S_{\B}\rangle
    - 2\langle S_{\B},\iota\rangle
    - \langle A\iota,\iota\rangle .
\end{align}
\end{corollary}

By combining Theorems \ref{Indices} and \ref{muF}, we obtain
\begin{corollary}{\cite[Proposition 4.7]{FP-GB-SM2021}}\label{mil = intmil}
If $\F$ is an arbitrary foliation on $(\C^2,0)$ with balanced divisor $\B$ and weighted balanced divisor $\mc B'$ then 
 \[\mu_0(\F)-\mu_0(\F,\mc B)=\langle -A^{-1}S_\F-(I+F^{-1})u,T_\F\rangle,\] 
 \[\mu_0(\F)-\mu_0(\F,\mc B')=\langle -A^{-1}S_\F-(I+F^{-1})u,C_\F\rangle,\]
 In particular
 \begin{enumerate}
 \item If $\F$ is of second type then $\mu_0(\F)=\mu_0(\F,\mc B)$.
 \item If $\F$ is CND then $\mu_0(\F)=\mu_0(\F,\mc B')$.
 \end{enumerate}
\end{corollary}

\begin{proof} From Theorems~\ref{muF} and~\ref{Indices} we obtain
\[\mu_0(\F)-\mu_0(\F,C)=
\langle -A^{-1}S_\F-(I+F^{-1})u,S_\F-S_C\rangle+\mathfrak{T}_{\F}(\mc I)-\mathfrak{T}_{\F}(C).\]
If we take $C=\mc B$ then $S_\F-S_C=T_\F$ and $\mathfrak{T}_{\F}(C)=\mathfrak{T}_{\F}(\mc I)$. If we take $C=\mc B'$ then $S_\F-S_C=C_\F$ and $\mathfrak{T}_{\F}(C)=\mathfrak{T}_{\F}(\mc I)$.
\end{proof}

Notice that the converses of assertions $(1)$ and $(2)$ in Corollary~\ref{mil = intmil} are not true as the following examples show:
 \begin{example}
The multiplicity one foliation $\F$ defined by $ydx+(y-x)dy$ is not of second type. In fact, it reduces after one blow-up having a single singular point which is a tangent saddle-node with Milnor number $2$. Thus, in this case $-A^{-1}=F=S_{\mc B}=T_\F=1$ and $S_\F=2$.  
According to 
Corollary~\ref{mil = intmil} we have  
\[
\mu_0(\F)-\mu_0(\F,\mc B)=\langle -A^{-1}S_\F-(I+F^{-1})u,T_\F\rangle=0.\]
\end{example}

\begin{example}
The foliation considered in Example~\ref{Dulac}, defined by $\omega=ydx-(2x+y^2)dy$ can be reduced with two blow-ups and have one non-degenerate singularity in the first component $E_1$ of the exceptional divisor  and one corner saddle-node  at $E_1\cap E_2$ so that it is not a CND foliation. Moreover,
$\mc B=\mc B'=\{y=0\}$, $S_{\mc B}=(1,0)^\mathsf{T}$, $T_\F=C_\F=(0,1)^\mathsf{T}$, 
$S_\F=u=(1,1)^\mathsf{T}$, $A=\left(\begin{array}{cc}-2 & 1\\ 1 & -1\end{array}\right)$, 
$F=\left(\begin{array}{cc}1 & 0\\ -1 & 1\end{array}\right)$ 
and consequently $-A^{-1}S_\F-(I+F^{-1})u=0$.
\end{example}

\section{Some applications}

In \cite[\S3.5]{FM} the {\bf polar excess index} of a foliation $\F$ on $(\C^2,0)$ with respect to an effective and reduced invariant divisor $C$ contained in a balanced divisor $\mc B$ of $\F$  is defined by 
\[\Delta_0(\F,C)=i_0(\mc P^\F,C)-i_0(\mc P^{\mc B},C),\]
where $\mc P^\F$ (resp. $\mc P^{\mc B}$) is a generic polar curve of $\F$ (resp. the foliation defined by a reduced equation of $\mc B$).

\begin{corollary}\label{Delta} Let $C$ be a $\F$-invariant reduced effective divisor. Then the polar excess index $\Delta_0(\F,C)$ of $\F$ with respect to $C$ can be computed as
\[\Delta_0(\F,C)=\mathfrak{T}_\F(C)+\langle -A^{-1}T_\F,S_C\rangle.\]
\end{corollary}
\begin{proof}
From \cite[Theorem 3.3]{FM} we obtain that
\[GSV_0(\F,C)=\Delta_0(\F,C)+\langle -A^{-1}(S_{\mc B}-S_C),S_C\rangle\]
and we compare with Theorem~\ref{Indices}.
\end{proof}

Since $\Delta_0(\F,C_1+C_2)=\Delta_0(\F,C_1)+\Delta_0(\F,C_2)$ we extend by linearity $\Delta_0$ for arbitrary $\F$-invariant reduced divisors.

\begin{remark}\label{SVK}
Using Seifert-Van Kampen theorem, if $\pi$ is a desingularization
of a reduced germ of curve $C:f_1\cdots f_m=0$ then 
the fundamental group of the complement of $C$ in a Milnor ball $\mb B$ is
generated by loops $\gamma_K$, where $ K$ varies in the set of irreducible components of $\pi^{-1}(C)=\bar C_1\cup\cdots\cup\bar C_m\cup E_1\cup\cdots E_n$, with the following relations
\[\prod_K\gamma_K^{K\cdot E_i}=\gamma_{E_i}^{E_i\cdot E_i}\prod_{K\neq E_i} \gamma_{K}^{K\cdot E_i}=1,\quad i=1,\ldots,n.\]
Then $H_1(\mb B\setminus C,\Z)=\bigoplus_{K}\Z c_K/\langle \sum_K K\cdot E_i c_K,i=1,\ldots,n\rangle=\Z c_1\oplus\cdots\oplus\Z c_m$, where $c_i=c_{\bar C_i}$ is a cycle of index $1$ around $f_i=0$ and $0$ around $f_j=0$ if $j\neq i$.
Thus 
\[(c_{E_1},\ldots, c_{E_n})A=-\sum_{i=1}^m(\bar C_i\cdot E_1,\ldots,\bar C_i\cdot E_n)c_i\]
and
\[(c_{E_1},\ldots, c_{E_n})=-\sum_{i=1}^m(\bar C_i\cdot E_1,\ldots,\bar C_i\cdot E_n)A^{-1}c_i.\]
Since $\frac{1}{2\pi\sqrt{-1}}\int_{c_i}\frac{d\pi^*f_i}{\pi^* f_i}=\frac{1}{2\pi\sqrt{-1}}\int_{\pi(c_i)}\frac{df_i}{f_i}=1$,
the  vector of vanishing orders of $\pi^*f_i$ along each $E_j$ is 
\begin{equation}\label{M}
M_{f_i}=\left(\frac{1}{2\pi\sqrt{-1}}\int_{c_{E_j}}\frac{d\pi^*f_i}{\pi^*f_i}\right)_j=-A^{-1}(\bar C_i\cdot E_j)_j=-A^{-1}S_{C_i}.
\end{equation}
By linearity, $M_f=M_{f_1\cdots f_m}=\sum_{i=1}^m-A^{-1}S_{C_i}=-A^{-1}S_C$.
\end{remark}

 We recover the results \cite[Proposition 7]{Brunella} and 
\cite[Th\'eor\`eme 3.3]{CL}.
\begin{proposition}\label{LC}
Let $\F$ be a non-dicritical foliation and let $C$ be the union of all its se\-pa\-ra\-tri\-ces. Then
$GSV_0(\F,C)=0$ if and only if $\F$ is a generalized curve.
\end{proposition}
\begin{proof}
If $\F$ is non-dicritical and $C$ is the union of all its separatrices then
\[GSV_0(\F,C)=
\langle T_\F,\underbrace{-A^{-1}S_C}_{M_C}\rangle+
\underbrace{\sum_{p\in SN(\bar\F)\setminus TSN(\bar\F)}(\mu_p(\bar\F)-1)}_{\ge 0\text{ and $=0\Leftrightarrow SN=TSN$}}.\]
Since $M_C$ is a vector whose all their entries are strictly positive (see Remark~\ref{SVK}), we deduce that $\langle T_\F,M_C\rangle\ge 0$ and $\langle T_\F,M_C\rangle=0$ if and only if $T_\F=0$.
Thus, we conclude that 
\[GSV_0(\F,C)=0\Leftrightarrow SN(\bar\F)=\emptyset\Leftrightarrow\F\text{ is generalized curve.}\]
\end{proof}

\begin{lemma}\label{SBT}
For every foliation $\F$ with balanced divisor $\mc B$ we have that all the entries of the vector $-A^{-1}S_{\mc B}$ are strictly positive.
In particular,
$\langle -A^{-1}S_{\mc B},T_\F\rangle\ge 0$
and equality holds if and only if $\F$ is of second type.
\end{lemma}
\begin{proof}
It suffices to prove the assertion for a minimal balanced divisor.
The proof of \cite[Proposition 3.5]{FM} in fact shows that for each dicritical component $D$ of valence $\ge 3$ and for  each dicritical separatrix $B$ in $\mc B_\infty$ meeting $D$ we can associate an isolated separatrix $\tilde B$  such that
$\nu_{p_i}(B)\le \nu_{p_i}(\tilde B)$ for all $i=0,1,\ldots,n-1$. Moreover the correspondence $B\mapsto \tilde B$ is injective.  
In fact, we order the dicritical components of the exceptional divisor having valence at least 3 by order of appearance $D_1,\ldots,D_d$. 
The first time that $D_1$ appears in the reduction process, it has valence $0$, $1$ or $2$ (corresponding, respectively, to the blow-up at $p_0=0$ itself, at a non-corner singularity or at a corner singularity). Therefore at least $\mr{val}_{D_1} - 2$ points of $D_1$ will be blown-up in the subsequent steps of the reduction process and to each one of them we can associate an isolated separatrix thanks to the generalization \cite{O-B} of the separatrix theorem \cite{CS}, see also \cite{Rosas}. Since the number of elements $B$ in $\mc B_\infty$ is $\mr{val}_{D_1} - 2$, we have the desired correspondece $B\mapsto \tilde B$ for the discritical separatrices $B$ transverse to the dicritical component $D_1$. The same procedure applies to $D_2$ and the remaining dicritical components, see Figure~\ref{dibujo}.
Still needs to see that $\nu_{p_i}(B)\le \nu_{p_i}(\tilde B)$.
Since $B$ is a dicritical separatrix once the dicritical component $D_1$ appears, the strict transform of $B$ does not pass through the  centers $p_{k},\ldots,p_{n-1}$ of the remaining blow-ups and consequently $\nu_{p_i}(B)=0\le \nu_{p_i}(\tilde B)$ for $i=k,\ldots,n-1$.
We consider the vectors $S_B=(0,\ldots,0,1)^\mathsf{T}$
and
$S_{\tilde B}$
 associated to the composition of blow-ups $\pi_{p_{0}}\circ\pi_{p_1} \circ\cdots\circ\pi_{p_{k-1}}$.
 The entries of $S_{\tilde B}$ are all $\ge 0$ and the last one is $\ge 1$ so that the entries of $S_{\tilde B}-S_B$ are all $\ge 0$.
  By Lemma~\ref{multiplicidades de curva},  all the entries of the vector $\nu(\tilde B)-\nu(B)=(F^{-1})^\mathsf{T}(S_{\tilde B}-S_B)$ are $\ge 0$ because this property is satisfied by the matrix $(F^{-1})^\mathsf{T}$ (see Corollary~\ref{weights} and Remark~\ref{F}).
  
This allows to write $\mc B_0 =\hat{\mc B}_0+\tilde{\mc B}_0$ and 
$\mc B=\hat{\mc B}_0+(\tilde{\mc B}_0-\mc B_\infty)$
with  $\hat{\mc B}_0$ and $\tilde{\mc B}_0$ effective and  $\nu_{p_i}(\tilde{\mc B}_0-\mc B_\infty)\ge 0$ for each $i=0,\ldots,n-1$. Moreover, $\hat{\mc B_0}\neq0$ because the valence of $D$ minus one is greater or equal than the number of dicritical separatrices in $\mc B_\infty$ meeting $D$. Thus, we have
\[-A^{-1}S_{\mc B}=-A^{-1}S_{\hat{\mc B}_0}-A^{-1}S_{\tilde{\mc B}_0-\mc B_\infty}=-A^{-1}S_{\hat{\mc B}_0}+F^{-1}(F^{-1})^\mathsf{T}S_{\tilde{\mc B}_0-\mc B_\infty}.\]
Since
all the entries of $F^{-1}$ are $\ge 0$ (see Remark~\ref{F} and Lemma~\ref{weights}), all the entries of the vector
$(F^{-1})^\mathsf{T}S_{\tilde{\mc B}_0-\mc B_\infty}=\nu(\tilde{\mc B}_0-\mc B_\infty)$ are non-negative
and  all the entries of the vector $-A^{-1}S_{\hat{\mc B}_0}$
are strictly positive because $\hat{\mc B}_0$ is a non-trivial effective divisor (see (\ref{M}) in Remark~\ref{SVK}).
\end{proof}

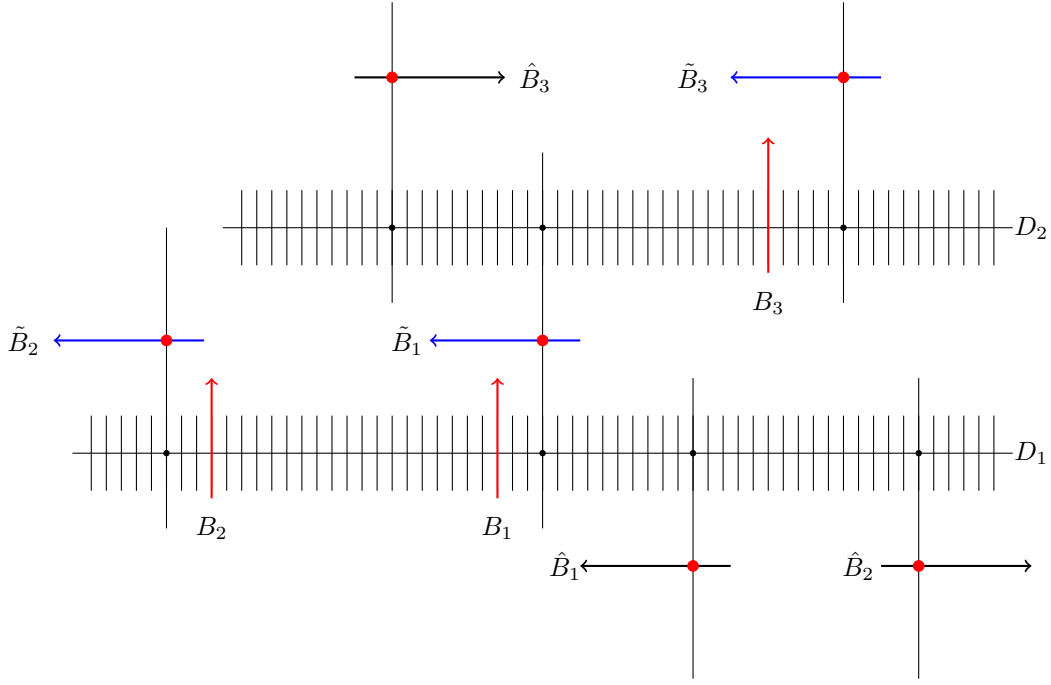
\begin{figure}[t]
\begin{tikzpicture}
 \node at (6.5,0) {$D_1$};
\draw (-6.25,0) to (6.25,0);
\draw (0,-1) to (0,4);
\draw (-5,-1) to (-5,3);
\draw (2,1) to (2,-3);
\draw (5,1) to (5,-3);
\foreach \k [evaluate=\k as \x using 0.2*\k] in {-30,...,30}{
    \draw (\x, -0.5) -- (\x, 0.5);
}
 \draw (-4.25,3) to (6.25,3);
 \node at (6.5,3) {$D_2$};
\foreach \k [evaluate=\k as \x using 0.2*\k] in {-20,...,30}{
    \draw (\x, 2.5) -- (\x, 3.5);
}
  \draw[->,blue,thick] (-4.5,1.5) to (-6.5,1.5);
  \draw[->,blue,thick] (.5,1.5) to (-1.5,1.5);
  \draw[->,thick] (4.5,-1.5) to (6.5,-1.5);
  \draw[->,thick] (2.5,-1.5) to (0.5,-1.5);
\draw (4,2) to (4,6);
\draw[->,blue,thick] (4.5,5) to (2.5,5);
\draw (-2,2) to (-2,6);
\draw[->,thick] (-2.5,5) to (-.5,5);
\draw[->,red,thick] (-.6,-.6) to (-.6, 1);
\node at (2,5) {$\tilde B_3$};
\node at (-.6,-1) {$B_1$};
\node at (-1.8,1.5) {$\tilde B_1$};
\draw[->,red,thick] (-4.4,-.6) to (-4.4, 1);
\node at (-4.4,-1) {$B_2$};
\node at (-6.9,1.5) {$\tilde B_2$};
\draw[->,red,thick] (3,2.4) to (3,4.2);
\node at (3,2) {$B_3$};
\node at (.3,-1.5) {$\hat B_1$};
\node at (4.2,-1.5) {$\hat B_2$};
\node at (-0.1,5) {$\hat B_3$};
\draw [fill] (0,0) circle [radius=1pt];
\draw [fill] (-5,0) circle [radius=1pt];
\draw [fill] (2,0) circle [radius=1pt];
\draw [fill] (5,0) circle [radius=1pt];
\draw [fill] (-2,3) circle [radius=1pt];
\draw [fill] (0,3) circle [radius=1pt];
\draw [fill] (4,3) circle [radius=1pt];
\draw [fill,red] (-2,5) circle [radius=2pt];
\draw [fill,red] (4,5) circle [radius=2pt];
\draw [fill,red] (-5,1.5) circle [radius=2pt];
\draw [fill,red] (0,1.5) circle [radius=2pt];
\draw [fill,red] (2,-1.5) circle [radius=2pt];
\draw [fill,red] (5,-1.5) circle [radius=2pt];
\end{tikzpicture}
\caption{Symbolic illustration of the correspondence $B_i\leadsto\tilde B_i$ between dicritical and isolated separatrices 
and the divisor $\hat{\mc B}_0=\hat B_1+\hat B_2+\hat B_3$.}\label{dibujo}
\end{figure}

Now we are able to give an extension of  \cite[Proposition 9]{Brunella}, see also  \cite[Theorem I]{FM}.

\begin{theorem}\label{Pi}
For an arbitrary foliation $\F$ with balanced divisor $\mc B$ we have that 
\begin{align*}
Var_0(\F,\mc B)-CS_0(\F,\mc B)&=\hphantom{2}\mathfrak{T}_\F(\mc B)+\hphantom{2}\langle -A^{-1}S_{\mc B},T_\F\rangle\\
&=\hphantom{2}\Delta_0(\F,\mc B)
{\ge}0
,\\
BB_0(\F)-Var_0(\F,\mc B)&=\hphantom{2}\mathfrak{T}_\F(\mc B)+\hphantom{2}\langle -A^{-1}S_{\mc B},T_\F\rangle+\langle -A^{-1}T_\F,T_\F\rangle\\
&=\hphantom{2}\Delta_0(\F,\mc B)+\|(F^{-1})^\mathsf{T}T_\F\|^2\ge 0,\\
BB_0(\F)-CS_0(\F,\mc B)&=
2\mathfrak{T}_\F(\mc B)+2\langle -A^{-1}S_{\mc B},T_\F\rangle+\langle -A^{-1}T_\F,T_\F\rangle\\
&=2\Delta_0(\F,\mc B)+\|(F^{-1})^\mathsf{T}T_\F\|^2\ge 0.
\end{align*}
Moreover, the following assertions are equivalent:
\begin{enumerate}[(i)]
\item $\F$ is a generalized curve (i.e. $\mathfrak{T}_\F(\mc B)=0$ and $T_\F=0$),
\item $Var_0(\F,\mc B)=CS_0(\F,\mc B)$,
\item $BB_0(\F)=Var_0(\F,\mc B)$,
\item $BB_0(\F)=CS_0(\F,\mc B)$.
\end{enumerate}
\end{theorem}
\begin{remark} The  equivalence $(i)\Leftrightarrow(iii)$ was proved in \cite[Theorem I]{FM}. 
\end{remark}

\begin{proof}
Using Theorem~\ref{Indices}, relation (\ref{VarCS}) and that $\mc B_\infty\cap E\cap\mr{Sing}(\bar\F)=\emptyset$ we get
\begin{align*}
Var_0(\F,\mc B)-CS_0(\F,\mc B)&=\sum_{p\in E\cap\bar{\mc B}} GSV_p(\bar\F,\mc B)+\langle -A^{-1}(S_\F-S_{\mc B})-\iota,S_{\mc B}\rangle\\
&=\langle -A^{-1}(S_\F-S_{\mc B}),S_{\mc B}\rangle+\mathfrak T_\F(\mc B)\\
&=\mathfrak T_\F(\mc B)+\langle \underbrace{S_\F-S_{\mc B}}_{T_\F},-A^{-1}S_{\mc B}\rangle=\Delta_0(\F,\mc B),
\end{align*}
where in the second equality we use formula (\ref{sumGSV}) and Corollary~\ref{Delta} in the last one. The fact, that $\Delta_0(\F,\mc B)\ge 0$ follows from Lemma~\ref{SBT}.

Let us prove now the third equality (the second  will follow from the first and the third).
Recall that if $p$ is a reduced singularity of a foliation germ $\F$ with (formal) separatrices $S=S_+\cup S_-$ then
\begin{align*}
BB_p(\F)&=\left\{\begin{array}{ll} \lambda+\frac{1}{\lambda}+2 & \text{if $p$ is non-degenerate with $\{CS_p(\F,S_\pm)\}=\{\lambda^{\pm 1}\}$,}
\\
2\mu+\lambda & 
\text{if $p$ is saddle-node with normal form $\omega_{\mu,\lambda}=x^\mu dy-y(1+\lambda x^{\mu-1})dx$.}
\end{array}\right.
\end{align*}
Moreover, $CS_0(\omega_{\mu,\lambda},x=0)=0$ and $CS_0(\omega_{\mu,\lambda},y=0)=\lambda$.

Let $\mc I$ be the isolated separatrices of $\F$ and $\bar{\mc I}\cap E=\{p_1,\ldots,p_M\}$ and $\lambda_j=CS_{p_j}(\bar\F,\bar{\mc I})$ for each $j=1,\ldots,M$. Let us denote by \( E^i_1, \ldots, E^i_k \) the invariant components of \( E \), and let \( q_1, \ldots, q_N \) be the intersection points between any two distinct components among them. According to Theorem~\ref{Indices},
\[BB_0(\F)=\sum_{p\in\mr{Sing}(\bar\F)} BB_p(\bar\F) +\langle -A^{-1}S_\F,S_\F\rangle-2\underbrace{\langle S_\F,\iota\rangle}_{\underbrace{\langle T_\F,\iota\rangle}_{\langle T_\F,u\rangle}+\langle S_{\mc B},\iota\rangle}-\langle A\iota,\iota\rangle\]
and
\[CS_0(\F,\mc B)=\sum_{p\in\mc I\cap E} CS_p(\F,\bar{\mc I})+\langle -A^{-1}S_{\mc B},S_{\mc B}\rangle.\]
Notice that
\begin{enumerate}[(a)]
\item
If $p_i\notin SN(\bar\F)$ then $BB_{p_i}(\bar\F)-CS_{p_i}(\bar\F,\bar{\mc I})={2}+CS_{p_i}(\bar\F,E^i_{p_i})$.
\item
If $p_i\in  TSN(\bar\F)\setminus CSN(\bar\F)$ then  $BB_{p_i}(\bar\F)-CS_{p_i}(\bar\F,\bar{\mc I})=2\mu_{p_i}(\bar\F)+CS_{p_i}(\bar\F,E^i_{p_i})$.
\item
If $p_i\in SN(\bar\F)\setminus TSN(\bar\F)$ then  $BB_{p_i}(\bar\F)-CS_{p_i}(\bar\F,\bar{\mc I})=2\mu_{p_i}(\bar\F)+CS_{p_i}(\bar\F,\bar{\mc I})-CS_{p_i}(\bar\F,\bar{\mc I})=2\mu_{p_i}(\bar\F)$.
\end{enumerate}
Denoting $X=\langle -A^{-1}T_\F,T_\F\rangle+2\langle -A^{-1}S_{\mc B}-u,T_\F\rangle$, we have that $BB_0(\F)-CS_0(\F,\mc B)-X$ is equal to
\begin{align*}
&\hphantom{=} \sum_{p\in\mr{Sing}(\bar\F)}BB_p(\bar\F)-\sum_{p\in E\cap\bar{\mc I}}CS_p(\bar\F,\bar{\mc I})-2\langle S_{\mc B},\iota\rangle-\langle A\iota,\iota\rangle\\
&=\sum_{i=1}^M(BB_{p_i}(\bar\F)-CS_{p_i}(\bar\F,\bar{\mc I}))+\sum_{j=1}^NBB_{q_j}(\bar\F)-2\langle S_{\mc B},\iota\rangle-\langle A\iota,\iota\rangle\\
&=\sum_{m=1}^k\left[\sum_{p_i\in E_m^i{\setminus SN(\bar\F)}}({2+}CS_{p_i}(\bar\F,E_m^i))+\sum_{p_i\in {E^i_m\cap T}SN(\bar\F)\setminus CSN(\bar\F)}(2\mu_{p_i}(\bar\F){+CS_{p_i}(\bar\F,E^i_m)})\right]\\
 & \hspace{5mm}
 +{\sum_{m=1}^k\sum_{p_i\in E^i_m\cap SN(\bar\F)\setminus TSN(\bar\F)}2\mu_{p_i}(\bar\F)}+2(N-|CSN(\bar\F)|)
{+\sum_{q_j\notin CSN(\bar\F)}(CS_{q_j}(E_{q_j}^1)+CS_{q_j}(E_{q_j}^2))}\\
&\hspace{5mm}
+\sum_{q_j\in CSN}(2\mu_{q_j}(\bar\F){+CS_{q_j}(\bar\F,E^w_{q_j}))}-2\langle S_{\mc B},\iota\rangle-\langle A\iota,\iota\rangle\\
&\stackrel{\text{(C-S)}}{=}\sum_{m=1}^k(E_m^i)^2+2\sum_{p\in SN(\bar\F)}\mu_p(\bar\F)+2N-2|CSN(\bar\F)|-(E_1^i+\cdots+E_k^i)^2
\\&\hspace{5mm}+2|\{p_i\}_{i=1}^M\setminus SN(\bar\F)|
-2\langle S_{\mc B},\iota\rangle
\\
&= 2\sum_{p\in SN(\bar\F)}(\mu_p(\bar\F)-1)=2(\mathfrak{T}_\F(\mc B)+\langle T_\F,u\rangle).
\end{align*}
because $|\{p_i\}_{i=1}^M\setminus SN(\bar\F)|=|\{p_i\}_{i=1}^M\setminus(SN(\bar\F)\setminus CSN(\bar\F))|=M-(|SN(\bar\F)|-|CSN(\bar\F)|)$, $\langle S_{\mc B},\iota\rangle=M$ and $\sum_{r\neq s}E_r^i\cdot E_s^i=2N$. Notice that in the equality (C-S) we have used the Camacho-Sad index theorem.

Finally, the equivalence between assertions (i)-(iv) is clear by the previous equalities and Lemma~\ref{SBT}.
\end{proof}

\section{About the order in the sequence of blow-ups}

Let $\Pi=(\pi_{p_0},\ldots,\pi_{p_{n-1}})$ be an ordered sequence of blow-ups $\pi_{p_i}$, $i=0,\ldots,n-1$ of centers \[p_i\in(\pi_{p_{0}}\circ\cdots\circ \pi_{p_{i-1}})^{-1}(p_0),\] with $p_0=0\in\C^2$. Denote $\pi_\Pi=\pi_{p_{0}}\circ\cdots\circ\pi_{p_{n-1}}$.
Notice that there may exist (admissible) permutations $\sigma$ of $\{0,1,\ldots,n-1\}$ (with $\sigma(0)=0$) such that $\pi_{\Pi^\sigma}=\pi_\Pi$, where $\Pi^\sigma=(\pi_{\sigma(0)},\ldots,\pi_{\sigma(n-1)})$.
Notice that the vectors $S_C^\Pi$ and $\nu_C^\Pi$ depends on $\Pi$ and not only on the composition $\pi_\Pi$.
Each admissible permutation $\sigma$ induces a (orthogonal) permutation matrix $\Sigma_\sigma$ 
such that $S_C^{\Pi^\sigma}=\Sigma_\sigma^{-1}S_C^\Pi$ and $\nu_C^{\Pi^\sigma}=\Sigma_\sigma^{-1}\nu_C^\Pi$. Consequently, $F_{\Pi^\sigma}=\Sigma_\sigma^\mathsf{T} F_\Pi\Sigma_\sigma$ and $-A_{\Pi^\sigma}^{\pm 1}=\Sigma_\sigma^\mathsf{T}(-A_\Pi^{\pm 1})\Sigma_\sigma$.

\begin{example} 
Let 
$\pi_0$ be the blow-up of the origin and $p_1,p_2$ two different points in the exceptional divisor $\pi_0^{-1}(0)$.  Then
$\pi_0\circ\pi_{p_1}\circ\pi_{p_2}=\pi_0\circ\pi_{p_2}\circ\pi_{p_1}$.

\end{example}

One can check that all the formulas we give are invariant by these admissible permutations.

On the other hand, once a composition of explosions has been performed, we can rearrange the components of the exceptional divisor and change all the vectors accordingly. 
More precisely, given an arbitrary ordering in the irreducible components of the exceptional divisor, we construct the corresponding intersection matrix $A'$ which satisfies
$A'=\Sigma^\mathsf{T}A\Sigma$ for a permutation matrix $\Sigma$. Then we define $F'=\Sigma^\mathsf{T}F\Sigma$ which is not necessarily lower triangular but satisfies $-(F')^\mathsf{T}F'=-A'$.
The formulas we have obtained apply without problems to the new arrangement.
For instance, the quantity $\mu_0(\F)-1-\mathfrak T_\F(\mc B)$ in Theorem~\ref{muF} can be computed~as
\begin{align*}
&\langle (-A'^{-1})S'_\F-(I+F'^{-1})u,S'_\F\rangle\\
=&\langle \Sigma^\mathsf{T}(-A^{-1})\Sigma \Sigma^{-1}S_\F-(I+\Sigma^\mathsf{T}F^{-1}\Sigma)u,\Sigma^{-1}S_\F\rangle\\
=&\langle \Sigma^\mathsf{T}(-A^{-1})\Sigma \Sigma^{-1}S_\F-\Sigma^\mathsf{T}(I+F^{-1})\Sigma u,\Sigma^{-1}S_\F\rangle\\
=&\langle \Sigma^\mathsf{T}(-A^{-1})\Sigma \Sigma^{-1}S_\F-\Sigma^\mathsf{T}(I+F^{-1})u,\Sigma^{-1}S_\F\rangle\\
=&\langle \Sigma^\mathsf{T}((-A^{-1})S_\F-(I+F^{-1})u),\Sigma^\mathsf{T}S_\F\rangle\\
=&\langle (-A^{-1})S_\F-(I+F^{-1})u,S_\F\rangle. 
\end{align*}

\begin{example}  Let us consider the minimal reduction of singularities of the foliation $\F$ defined by $d(y^2+x^3)$. The exceptional divisor has three irreducible components $E_1,E_2,E_3$ listed in order of appearance.
Consider the rearrangement $E_1'=E_2$, $E_2'=E_3$ and $E_3'=E_1$.
The intersection matrices associated to the two orderings and the corresponding permutation matrix are
\[A=\left(\begin{array}{ccc}
-3 & 0 & 1 
\\
 0 & -2 & 1 
\\
 1 & 1 & -1 
\end{array}\right),\quad A'=\left(\begin{array}{ccc}
-2 & 1 & 0 
\\
 1 & -1 & 1 
\\
 0 & 1 & -3 
\end{array}\right),\quad \Sigma=\left(\begin{array}{ccc}
0 & 0 & 1 
\\
 1 & 0 & 0 
\\
 0 & 1 & 0 
\end{array}\right).
\]
There is a unique balanced divisor 
 $\mc B=C=\{y^2+x^3=0\}$ and the vectors  $S_{\F}=S_C$ and $S_\F'=S_C'$ are given by
\[S_C=(0,0,1)^\mathsf{T},\quad S'_C=\Sigma^\mathsf{T}S_C=(0,1,0)^\mathsf{T}.\]
Finally, the proximity matrices are
\[F=\left(\begin{array}{ccc}
1 & 0 & 0 
\\
 -1 & 1 & 0 
\\
 -1 & -1 & 1 
\end{array}\right),\quad F'=\Sigma^\mathsf{T}F\Sigma=\left(\begin{array}{ccc}
1 & 0 & -1 
\\
 -1 & 1 & -1 
\\
 0 & 0 & 1 
\end{array}\right).
\]
Notice that $F'$ is no longer lower triangular.
\end{example}

\end{document}